\documentclass[12pt]{amsart}

\usepackage[utf8]{inputenc}
\usepackage[lite, alphabetic, nobysame]{amsrefs}
\usepackage{amsmath,mathtools}
\usepackage{amssymb}
\usepackage{srcltx}
\usepackage{xcolor}
\usepackage{xspace}
\usepackage{enumerate}
\usepackage{geometry}
\usepackage{marginnote}

\usepackage{float}
\usepackage{tikz}
\usepackage{tikz-cd}
\usepackage{wrapfig}
\usetikzlibrary{positioning}

\usepackage{amsmath, amsthm}
\usepackage{amssymb}
\usepackage{todonotes}

\usepackage{cite}

\usepackage{tikz}
\usetikzlibrary{positioning}

\theoremstyle{definition}

\newtheorem{thm}{Theorem}
\newtheorem*{thm*}{Theorem}

\newtheorem{defn}[thm]{Definition}

\newtheorem{claim}[thm]{Claim}

\newtheorem{remark}[thm]{Remark}
\newtheorem{fact}[thm]{Fact}

\newtheorem{conj}[thm]{Conjecture}

\renewcommand{\subset}{\subseteq}

\newcommand\R{\mathbb{R}}
\newcommand\N{\mathbb{N}}

\newcommand{\set}[2]{ \left\{ #1 :\, #2 \right\} }
\newcommand{\seqq}[2]{ \left( #1 :\, #2\right) }

\title{A counterexample regarding an equivalence relation on a product space}
\author{Assaf Shani}
\address{Department of Mathematics and Statistics, Concordia University, Montreal, QC  H3G 1M8, Canada}
\email{assaf.shani@concordia.ca}
\urladdr{https://sites.google.com/view/assaf-shani/}
\thanks{Research partially supported by NSF grant DMS-2246746 and NSERC grant RGPIN-2024-05827.}
\thanks{I would like to thank the anonymous referee for thoughtful and helpful comments.}
\date{\today}
\begin{document}
\maketitle
\begin{abstract}
    We present a counterexample to Conjecture~14.1.6 from \cite{Kano08}, regarding Borel equivalence relations on product spaces.
\end{abstract}

This paper concerns the study of equivalence relations on Polish spaces up to Borel reducibility.
Given equivalence relations $E$ and $F$ on Polish spaces $X$ and $Y$ respectively, a map $f\colon X\to Y$ is said to be a \textbf{reduction} of $E$ to $F$ if for any $x_1,x_2\in X$, 
\begin{equation*}
    x_1\mathrel{E}x_2\iff f(x_1)\mathrel{F}f(x_2).
\end{equation*}
We say that $E$ is \textbf{Borel reducible} to $F$, denoted ${E}\leq_B{F}$ if there is a Borel measurable function which is a reduction of $E$ to $F$. In this case, we think of $E$ as no more complicated than $F$. Borel reducibility is the most central concept in the study of equivalence relations on Polish spaces. Say that $E$ and $F$ are \textbf{Borel bireducible}, denoted $E\sim_B F$, if $E\leq_B F$ and $F\leq_B E$. We write $E<_B F$ (a strict reduction) if a Borel reduction exists in only that direction. 
An equivalence relation $E$ on a Polish space $X$ is \textbf{Borel} if $E$ is a Borel subset of $X\times X$, with the product topology.

A central motivation for the field is to study the complexity of various classification problems in mathematics. Generally speaking, separable mathematical objects can be coded as members of some Polish space. Natural notions of equivalence, such as isomorphism between countable graphs, isometry between separable metric spaces, or homeomorphism between compact metric spaces, can then be seen as equivalence relations on Polish spaces. These are generally analytic, and sometiems Borel. The focus on Borel measurable reductions is an attempt to consider only maps coming from ``natural mathematical practice''. The particular choice of ``Borel measurability'' is strongly supported by theoretical and experimental evidence, as well as the success of the field since its inception with the papers \cite{HKL90, FS89}. For more background the reader is referred to the expository books~\cite{Gao09,Kano08}, and the surveys \cite{Hjorth-Kechris-recent-developments-2001, Kechris99-Survey-NewDirections, Kechris_1997-classification-problems, kechris_tucker-drob_2012, Foreman-Borel-reduction, MottoRos-Survey21}.

Given equivalence relations $F_1, F_2$ on spaces $Y_1, Y_2$ respectively, define $F_1\times F_2$ on $Y_1\times Y_2$ as the pointwise relation: $(y_1,y_2), (y'_1,y'_2)\in Y_1\times Y_2$ are $F_1\times F_2$-related if $y_1\mathrel{F_1}y'_1$ and $y_2\mathrel{F_2}y'_2$.

In this note we provide a counterexample to the following conjecture, attributed to Zapletal in \cite{Kano08}. 

\begin{conj}[Conjecture~14.1.6 in~{\cite{Kano08}}]\label{conj: kanovei-zapletal}
    Let $X,Y$ be Polish spaces, $P\subset X\times Y$ a Borel set, $F$ a Borel equivalence relation on $X$ and $E$ a Borel equivalence relation on $P$ so that for any $(x,y),(x',y')\in P$,
    \begin{equation*}
        (x,y)\mathrel{E}(x',y') \implies x\mathrel{F}x'.
    \end{equation*}
    Assume that $G$ is a Borel equivalence relation so that for any $x_0\in X$,
    \begin{equation*}
        E\restriction \set{(x,y)\in P}{x\mathrel{F}x_0} \mathrel{\leq_B} {G}.
    \end{equation*}
    Then ${E}\mathrel{\leq_B}{F\times G}$.
\end{conj}

Roughly speaking, the conjecture says that if $E$ extends $F$ on the first coordinate, and on every fiber $E$ is no more complicated than $G$, then $E$ is no more complicated than $F\times G$.

Two particularly important families of Borel equivalence relations are the smooth equivalence relations and the countable equivalence relations. (See Chapters 5.4 \& 7 in~\cite{Gao09}.)

An equivalence relation $G$ is said to be \emph{smooth} if $G$ is Borel reducible to $=_\R$, where $=_\R$ is the equality relation on the Polish space $\R$. 
A Borel equivalence relation $F$ on a Polish space $Y$ is \emph{countable} if each $F$-class is countable.

The motivation for this conjecture is as follows. The conjecture is true in the case where $F$ is a countable Borel equivalence relation and $G$ is smooth (see~\cite[Theorem~14.1.1]{Kano08}). This fact is essential in the proof of a dichotomy, due to Hjorth and Kechris~\cite{Hjorth-Kechris-recent-developments-2001}, for the equivalence relation $E_3$ (see \cite[Chapter~14]{Kano08}). 

Our counterexample will use the Friedman-Stanley jumps, defined as follows.
Given a space $X$, let $X^\N$ be the space of all sequence $x = \seqq{x(n)}{n\in\N}$, where $x(n) \in X$, equipped with the product topology. If $X$ is Polish, then $X^\N$ is Polish.
\begin{defn}[Friedman-Stanley~\cite{FS89}]
        Let $E$ be an equivalence relation on a Polish space $X$. Define $E^+$ on the space $X^\N$ by
        \begin{equation*}
            x\mathrel{E^+}y \iff \forall n \exists m (x(n)\mathrel{E}y(m)) \mathrel{\&} \forall n \exists m (y(n) \mathrel{E} x(m)),
        \end{equation*}
        that is, $\set{[x(n)]_E}{n\in\N}=\set{[y(n)]_E}{n\in\N}$.
    \end{defn}
Note that if $E$ is Borel, then $E^+$ is Borel. Friedman and Stanley proved the following (see~\cite[8.3.6]{Gao09}).
\begin{fact}[Friedman-Stanley~{\cite{FS89}}]\label{fact: FS jump operator}
    For a Borel equivalence relation $E$, $E <_B E^+$.
\end{fact}

Define $F$ on $X = \R^\N$ by $ F = (=_{\R})^+$. That is, $x\mathrel{F}y$ if $x$ and $y$ enumerate the same sets of reals.

Next we define the equivalence relation $E$. Let $2^\N$ be the space of binary sequences, given the product topology, where the set $2 = \{0,1\}$ is given the discrete topology. 
Consider the Polish space $\R^\N \times (2^\N)^\N$ with the product topology. 
Given $(x,y)\in \R^\N \times (2^\N)^\N$, define $A^{(x,y)}_n = \set{x(m)}{y(n)(m)=1}$ for $n\in\N$. That is, we view each $y(n)$ as a binary sequence ``carving out'' a subset of the set enumerated by $x$. 

\begin{equation*}
       \begin{array}{ccccc}
\vdots & \vdots & \vdots & \vdots\\
{\color{yellow} \ast} & 1 & 0 & 1 & \ldots\\
{\color{green} \ast} & 1 & 1 & 0 & \ldots\\
{\color{red} \ast} & 0 & 1 & 1 & \ldots\\
{\color{blue} \ast} & 0 & 1 & 0 & \ldots\\
x & y(0) & y(1) & y(2)
\end{array}\mapsto
\begin{array}{cccc}
\vdots & \vdots & \vdots\\
{\color{yellow} \ast} & - & {\color{yellow} \ast} & \ldots\\
{\color{green} \ast} & {\color{green} \ast} & - & \ldots\\
- & {\color{red} \ast} & {\color{red} \ast} & \ldots\\
- & {\color{blue} \ast} & - & \ldots\\ 
A^{(x,y)}_0 & A^{(x,y)}_1 & A^{(x,y)}_2
\end{array}
\end{equation*}

Define $P\subset \R^\N \times (2^\N)^\N$ as the set of all $(x,y)$ such that:
\begin{enumerate}
    \item $\forall m\exists k (y(k)(m)=1)$;
    \item $\forall k \exists m (y(k)(m)=1)$;
    \item $\forall k,l_1,l_2(x(l_1)\mathrel{=}{}x(l_2)\rightarrow y(k)(l_1)=y(k)(l_2))$.
\end{enumerate}
That is: (1) each $x(m)$ appears in $A^{(x,y)}_k$ for some $k$, (2) each $A^{(x,y)}_k$ is not empty, and (3) if $x$ enumerates the same item twice, on indices $l_1$ and $l_2$, then each binary sequence $y(k)$ agrees on $l_1$ and $l_2$. Note that $P$ is a Borel set. In fact it is a dense $G_\delta$ subset of $\R^\N \times (2^\N)^\N$.

Finally, we define an equivalence relation $E$ on $P$ by
\begin{equation*}
    (x,y) \mathrel{E} (x',y') \iff \set{A^{(x,y)}_n}{n\in\N} = \set{A^{(x',y')}_n}{n\in\N},
\end{equation*}
that is, $(x,y)$ and $(x',y')$ code the same sets of subsets of $\R$. Note that $E$ is Borel.

\begin{remark}
    For $(x,y)\in P$, property (1) in the definition of $P$ implies that $\set{x(m)}{m\in\N} = \bigcup \set{A^{(x,y)}_n}{n\in\N}$. It follows that $E$ and $F$ satisfy the property
    \begin{equation*}
        (x,y)\mathrel{E}(x',y') \implies x\mathrel{F}x'.
    \end{equation*}
\end{remark}

Let $=_{2^\N}$ be the equality relation on $2^\N$, and let $G = (=_{2^\N})^+$. Note that $G$ and $F$ are Borel bireducible (as there is a Borel bijection between $\R$ and $2^\N$, see~\cite[1.4.3]{Gao09}).
\begin{claim}\label{claim: E fibers red to G}
    For any $x_0\in \R^\N$,
        \begin{equation*}
        E\restriction \set{(x,y)\in P}{x\mathrel{F}x_0} \mathrel{\leq_B} {G}.
    \end{equation*}
\end{claim}
\begin{proof}
    Fix $x_0\in \R^\N$. Define $f\colon \set{(x,y)\in P}{x\mathrel{F}x_0} \to (2^\N)^\N$ by
    \begin{equation*}
    \begin{split}
    f(x,y)(n)(k) = i & \iff (\forall k'\in\N)(x(k')=x_0(k) \implies y(n)(k')=i) \\ & \iff (\exists k'\in \N)(x(k')=x_0(k) \wedge y(n)(k')=i). 
    \end{split}    
    \end{equation*}
The equivalence between the definitions above follows from clause (3) in the definition of $P$. Note that $f(x,y)(n)(k)$ is well defined by clause (2). 

What $f$ does is the following. The $F$-class of $x_0$ are all sequences which enumerate the same set as $x_0$. Using a fixed enumeration of this set, via $x_0$, we code subsets of it, $A^{(x,y)}_n$, as binary sequences by identifying $k$ with $x_0(k)$. In particular, $f(x_0,y) = y$ for $(x_0,y)\in P$. We claim that $f$ is a reduction of $E\restriction \set{(x,y)\in P}{x\mathrel{F}x_0}$ to $G$.

The main point is the following. Given $(x,y), (x',y')\in P$ with $x\mathrel{F} x_0 \mathrel{F} x'$,
\begin{equation*}
(\star)\hspace{5mm} \forall n,m\in\N\, (A_n^{(x,y)} = A_m^{(x',y')} \iff f(x,y)(n) = f(x',y')(m)).
\end{equation*}
This implies that $(x,y)\mathrel{E}(x',y') \iff f(x,y)\mathrel{G}f(x',y')$, and so $f$ is the desired reduction.

To see $(\star)$, assume first that $A_n^{(x,y)} = A_m^{(x',y')}$. For each $k\in\N$ fix $k_1, k_2$ so that $x(k_1) = x_0(k) = x'(k_2)$. Then
\begin{equation*}
    f(x,y)(n)(k) = y(n)(k_1) = y'(m)(k_2) = f(x',y')(m)(k),
\end{equation*}
where the middle equality follows from clause (3) in the definition of $P$ and the assumption $A_n^{(x,y)} = A_m^{(x',y')}$. We conclude that $f(x,y)(n) = f(x,y)(m)$. 

Next, assume that $f(x,y)(n) = f(x',y')(m)$. For any $k_1$ so that $y(n)(k_1) = 1$ (and so $x(k_1)\in A_n^{(x,y)}$) we show that $x(k_1)\in A_m^{(x',y')}$. Fix $k,k_2$ so that $x(k_1) = x_0(k) = x'(k_2)$. Then
\begin{equation*}
    y'(m)(k_2) = f(x',y')(m)(k) = f(x,y)(n)(k) = y(n)(k_1) = 1,
\end{equation*}
and so $x(k_1) = x'(k_2) \in A_m^{(x',y')}$. It follows that $A_n^{(x,y)} \subset A_m^{(x',y')}$. By symmetry we also get $A_m^{(x',y')} \subset A_n^{(x,y)}$, and so $A_n^{(x,y)} = A_m^{(x',y')}$.
\end{proof}

\begin{thm}
The equivalence relations $E, F, G$ provide a counterexample to Conjecture~\ref{conj: kanovei-zapletal}.
\end{thm}

\begin{remark}
    The equivalence relation $E$ is Borel bireducible with the `second Friedman-Stanley jump', $((=_{\R})^+)^+$.
    The equivalence relation $E$ was studied in~\cite{Shani-FS-homomorphisms}, particularly with respect to Baire category techniques. Note that $((=_{\R})^+)^+$ is defined on the space $(\R^\N)^\N$ and therefore does not seem to provide a counterexample to Conjecture~\ref{conj: kanovei-zapletal}. The different presentation of $((=_{\R})^+)^+$, as $E$ on $\R^\N \times (2^\N)^\N$, is therefore crucial here. 
\end{remark}
We include a proof that $((=_{\R})^+)^+ \leq_B E$, which we use below. The idea is simple. Given $z\in (\R^\N)^\N$, each $z(n)\in \R^\N$ codes a subset of $\R$. We will send $z$ to a pair $(x,y)\in P$ so that $x$ enumerates the set of all reals appearing in the sets $z(n)$, and $y(n)$ carves out the subset $z(n)$ of $x$.

Fix a bijection $e\colon \N\times\N\to \N$. Define $x(e(i,j)) = z(i)(j)$. Define $y(n)(e(i,j)) = 1$ if there is some $k$ so that $z(n)(k) = z(i)(j)$, and $y(n)(e(i,j))=0$ otherwise. The map $z\mapsto (x,y)$ is a Borel reduction of $((=_\R)^+)^+$ to $E$.

We will also use the  well known fact that $(=_\R)^+ \times (=_\R)^+ \sim_B (=_\R)^+$. For the non-trivial direction, fix an injective Borel map $\iota\colon \R\times \{0,1\} \to \R$ and consider the map sending $(x,y)\in \R^\N \times \R^\N$ to $z\in\R^\N$ so that $z(2n) = \iota(x(n),0)$ and $z(2n+1) = \iota(y(n),1)$. Then $(x,y)\mapsto z$ is is a Borel reduction of $(=_\R)^+ \times (=_\R)^+$ to $(=_{\R})^+$.

Now Conjecture~\ref{conj: kanovei-zapletal} would imply that $E \leq_B F \times G$, and so
\begin{equation*}
    ((=_{\R})^+)^+ \leq_B E \leq_B F \times G \leq_B (=_\R)^+ \times (=_\R)^+ \leq_B (=_\R)^+,
\end{equation*}
a contradiction, as $((=_{\R})^+)^+ \not \leq_B (=_\R)^+$ by Fact~\ref{fact: FS jump operator}. 

\bibliographystyle{alpha}
\bibliography{bibliography}
\end{document}